\documentclass[12pt]{article}

\usepackage[citecolor=blue]{hyperref}
\usepackage[USenglish]{babel}
\usepackage{csquotes}
\usepackage[a4paper,margin=2.5cm]{geometry}

\usepackage[ocgcolorlinks]{ocgx2}

\usepackage{amsmath,mathtools}
\usepackage{amssymb,amsfonts,mathrsfs}

\usepackage{amsthm,thmtools}
\makeatletter
\@ifundefined{newcounteralias}{}{
    \renewcommand\thmt@autorefsetup{\@xa\def\csname\thmt@envname autorefname\@xa\endcsname\@xa{\thmt@thmname}}
}
\makeatother
\usepackage{reptheorem}

\usepackage[
    backend = biber,
    style = alphabetic,
    sorting = nyt,
    giveninits=true,
    maxbibnames=5,
    urldate=iso,
    seconds=true
]{biblatex}
\DeclareNameAlias{default}{family-given}
\newtheoremstyle
    {normalsl}
    {}{}{\slshape}{}{\bfseries}{.}{ }{}

\theoremstyle{normalsl}

\newtheorem{proposition}{Proposition}[section]
\newtheorem{lemma}[proposition]{Lemma}
\newtheorem{theorem}[proposition]{Theorem}
\newtheorem{claim}[proposition]{Claim}
\newtheorem{corollary}[proposition]{Corollary}

\theoremstyle{definition}

\newtheorem{definition}[proposition]{Definition}

\theoremstyle{remark}

\newtheorem{remark}[proposition]{Remark}

\title{Hamiltonian Floer Theory for Quantum Electrodynamics up to First Order in \texorpdfstring{\(\hbar\)}{ℏ}}
\author{Oliver Fabert\thanks{Vrije Universiteit Amsterdam} \and Jesse Straat\footnotemark[\value{footnote}]}
\date{August 20, 2026}

\begin{document}

\maketitle

\begin{abstract}
    We use Hamiltonian Floer theory to prove a cuplength result about the existence of periodic solutions of particle-field systems with a Gaussian random field. As a concrete model we study stochastic electrodynamics, which approximates quantum electrodynamics up to first order in \(\hbar\), and is even exact in the case of low-order Hamiltonians or when the particles are treated classically.
\end{abstract}
\tableofcontents
\section{Introduction}
Hamiltonian Floer theory is the fundamental tool in symplectic geometry used to prove dynamical results in classical mechanics, such as the existence of periodic solutions; for a general reference, see, e.g., \cite{audinMorseTheoryFloer2014}, and for the important case of cotangent bundles, see \cite{cieliebakPseudoholomorphicCurvesPeriodic1994} and the references therein. It is the goal of this paper to explore how Hamiltonian Floer theory can be used to study quantum dynamics, at least to some low order approximation in the Planck constant \(\hbar\).\par
The model we are studying is called stochastic electrodynamics (SED) and is considered to be the closest classical approximation to quantum electrodynamics, the quantum theory of the electromagnetic interaction. One original motivation is to understand the stability of the atom. From classical physics alone, in the Rutherford model, it follows that an electron, circling the nucleus as in the Kepler model of planetary motion, will fall into the nucleus in a very short time. Indeed, solving the coupled Maxwell--Lorentz equations, one finds that the electron, as an accelerating charge, emits electromagnetic radiation and hence loses energy until it comes to rest, colliding with the nucleus. Historically, this problem was solved using quantum theory, leading to the Bohr--Rutherford model (\cite{bohrConstitutionAtomsMolecules1913,bohrConstitutionAtomsMolecules1913a,bohrConstitutionAtomsMolecules1913b}) and the quantum mechanical atom (\cite{schroedingerQuantisierungAlsEigenwertproblem1926,schroedingerQuantisierungAlsEigenwertproblem1926a,schroedingerQuantisierungAlsEigenwertproblem1926b,schroedingerQuantisierungAlsEigenwertproblem1926c}), but stochastic electrodynamics takes a classical approach.\par
The key idea of stochastic electrodynamics is to treat the atom classically but in addition employ the vacuum fluctuations of the electromagnetic field --- that is, the quantum ground state of the electromagnetic field, also called the zero-point field --- to establish atomic stability. It can be viewed as a classical Gaussian random field, where the amplitude in each frequency is distributed normally with expectation value zero and standard deviation on the order of \(\hbar\), the average energy agreeing with the Casimir effect. The ground state of the atom is then characterized as an equilibrium in which the energy lost by the electron through radiation is precisely compensated (in the mean) by the energy absorbed from the zero-point field. \par
Until now, there is no mathematical proof showing whether this classical stochastic approach can really provide an alternative approach to atomic stability, though a heuristic argument is given in \cite{boyerRandomElectrodynamicsTheory1975}. However, the underlying idea has already been successfully applied for computing the ground state energy of the quantum harmonic oscillator in \cite{marshallRandomElectrodynamics1963}. Some researchers still work on establishing stochastic electrodynamics as an underlying theory behind quantum electrodynamics (see, e.g., \cite{delapenaEmergingQuantumPhysics2015,cettoQuantumMechanicsPhysical2025}). In this paper, we follow the more modest perspective of viewing stochastic electrodynamics as the closest (known) classical approximation to quantum electrodynamics, as in \cite{boyerStochasticElectrodynamicsClosest2019}.\par
In the recent paper \cite{santosAnalogyStochasticElectrodynamics2022}, the author uses the Weyl--Wigner approach to quantum theory to prove that the dynamics of a charged particle in quantum electrodynamics agrees with the dynamics of a charged particle in stochastic electrodynamics up to the first order in \(\hbar\). Indeed, the Moyal bracket --- the quantum extension of the classical Poisson bracket --- differs from the Poisson bracket only by terms of order \(\hbar^2\) and higher, so the dynamics up to first order are equivalent, assuming equal initial conditions for the field and for the particle. Since the Hamiltonian is quadratic with respect to the field coordinates, it moreover follows that the quantized field still satisfies the inhomogeneous Maxwell equations. In particular, its solution is the sum of the particular solution of the inhomogeneous system given by the charged particle dynamics and the solution of the homogeneous system given by the initial conditions. Instead of setting the homogeneous part equal to zero as in classical physics, one identifies the initial condition of the field with the zero-point field. Since the latter is of order \(\hbar\), it still remains visible after truncating all terms of order \(\hbar^2\) or higher. In other words, the dynamics in stochastic electrodynamics and in quantum electrodynamics agree up to first order in \(\hbar\) once the initial conditions for the particles are the same. Furthermore, the dynamics agree up to all orders with quantum models where the Hamiltonian is at most quadratic, such as in the case of the harmonic oscillator, or when the charged particle is considered fully classical (while the electromagnetic field remains fully quantized).\par
In this paper, we combine the work in \cite{fabertCuplengthEstimatesPeriodic2023} on periodic solutions of particle-field systems with the work in \cite{fabertCuplengthEstimatesTimeperiodic2024} on time-periodic measures of stochastic Hamiltonian systems to prove topological lower bounds on the number of time-periodic solutions in stochastic electrodynamics. More precisely, we consider the stochastic particle-field system given by the coupled Maxwell--Lorentz equations with the zero-point field. Note that, instead of imposing initial conditions for the particles, they are fixed by the goal of finding periodic solutions. The initial condition for the field is still given a priori by the zero-point field, i.e., the quantum ground state of the electromagnetic field. In order to be able to establish the necessary \(C^0\)-bounds, we assume that the charged particles sit in a periodic space, i.e., a torus. We further assume that the ratio between the space period \(L\) and the time period of the time-periodic solutions \(T\) is sufficiently generic. Moreover, we consider a limit in which magnetic forces are neglected, which may be understood as a low-speed (or nonrelativistic) limit. This stochastic model is a special case of what we will call a \emph{stochastic particle-field model}, on which we prove the following cuplength result.
\repthm{thm:cuplengthnonmagnetic}\par
This paper is organized as follows: first, we introduce the physical background in \autoref{sec:physicalbackground} by discussing Weyl--Wigner quantization and stochastic electrodynamics. Along the way, we fix our Hamiltonian treatment of electrodynamics in both Coulomb and Feynman--'t Hooft gauge. In \autoref{sec:nonmagneticlimit}, we define the stochastic particle-field model and prove the cuplength estimate in our main theorem. In particular, we find it describes a nonmagnetic limit of Feynman--'t Hooft electrodynamics, and perfectly models nonmagnetic electrodynamics with a quantized electric field. Finally, in \autoref{sec:C0bounds}, we discuss how our proof strategy needs to be adapted for stochastic electrodynamics beyond the nonmagnetic limit and discuss the role of regularization. 
\section{Stochastic approximations of quantum mechanics}\label{sec:physicalbackground}
We use this section to introduce how stochastic methods can be used to approximate quantum mechanics up to first order in \(\hbar\). To this end, we must first discuss the Weyl--Wigner model of quantum mechanics. We also introduce Hamiltonian models of electrodynamics to be used in the rest of this paper.\par
This section is considered a mathematical adaptation of existing physical theory. References to physical literature are provided where necessary.
\subsection{Weyl--Wigner quantization}
In this subsection, we introduce Weyl--Wigner quantization and discuss some established theorems without proofs. An interested reader can find more information in a comprehensive source, such as \cite{curtrightConciseTreatiseQuantum2014}. The implementation of Weyl--Wigner quantization for specifically (free) quantum electrodynamics is treated in \cite{santosQuantumElectromagneticField2024}.\par
Consider symplectic Euclidean space \((\mathbb{R}^{2n},\omega)\) with Darboux coordinates \((x_i,p_i)_{i=1}^n\). We equip it with a smooth time-dependent Hamiltonian \(H\colon\mathbb{R}^{2n}\times\mathbb{R}\to\mathbb{R}\). The Hamiltonian determines the dynamics of \(\mathbb{R}^{2n}\) and the functions therein. In classical mechanics, the time evolution of a smooth function \(f\colon\mathbb{R}^{2n}\to\mathbb{R}\) is given by the \emph{Liouville equation}
\begin{equation}
    \frac{\mathrm{d}}{\mathrm{d}t}f(x(t),p(t)) = \{H,f\}_P \coloneqq \sum_{i=1}^n \left(\frac{\mathrm{d}H}{\mathrm{d}p_i}\frac{\mathrm{d}f}{\mathrm{d}x_i} - \frac{\mathrm{d}H}{\mathrm{d}x_i}\frac{\mathrm{d}f}{\mathrm{d}p_i}\right),
\end{equation}
where \(\{-,-\}_P\) is called the \emph{Poisson bracket}. In (matrix) quantum mechanics, one typically replaces functions with operators and the Poisson bracket with a commutator. However, the Weyl approach instead leaves the functions as is and changes the commutator to get an equivalent algebra. In the following, \(\mathcal{D}\) is the space of compactly supported smooth functions (test functions), and \(\mathcal{D}'\) the space of distributions, i.e., the dual of \(\mathcal{D}\) with respect to the inner product on \(L^2\).
\begin{theorem}[Schwartz kernel theorem \cite{schwartzTheorieNoyaux}]
    There exists a linear homeomorphism between bounded operators \(\hat{f}\colon\mathcal{D}(\mathbb{R}^n)\to\mathcal{D}'(\mathbb{R}^n)\) and \(f\) in \(\mathcal{D}'(\mathbb{R}^{2n})\) (``symbols'') given by
    \begin{equation}
        \hat{f}\psi(x) = \frac{1}{(2\pi\hbar)^n}\int_{\mathbb{R}^{2n}} f\left(\frac{x+y}{2},p\right)e^{\frac{i}{\hbar}p\cdot(x-y)}\psi(y)\ \mathrm{d}y\mathrm{d}p.
    \end{equation}
\end{theorem}
This theorem essentially allows us to replace the operator algebra with an equivalent function algebra. The following properties of the linear homeomorphism follow from writing out.
\begin{proposition}
    Consider the homeomorphism induced by the Schwartz kernel theorem. The Hilbert--Schmidt norm is given by
    \begin{equation}
        \lVert\hat{f}\rVert_{HS} = \frac{1}{(2\pi\hbar)^{n/2}}\lVert f\rVert_{L^2}
    \end{equation}
    (in particular, any \(L^2\) function induces a compact operator),
    the commutator is mapped to \(i\hbar\) times the \emph{Moyal bracket},
    \begin{equation}
        [\hat{f},\hat{g}]\mapsto i\hbar\{f,g\}_M = \left.2i\sin\left(\frac{\hbar}{2}(\nabla_q\cdot\nabla_x - \nabla_p\cdot\nabla_y)\right)f(x,p)g(y,q)\right|_{x=y,p=q},
    \end{equation}
    and the probability distributions of \(\hat{f}\) in position, respectively momentum, space are given by
    \begin{equation}
        \langle x\rvert\hat{f}\lvert x\rangle = \frac{1}{(2\pi\hbar)^n}\int_{\mathbb{R}^n}f(x,p)\ \mathrm{d}p,\qquad \langle p\rvert\hat{f}\lvert p\rangle = \frac{1}{(2\pi\hbar)^n}\int_{\mathbb{R}^n}f(x,p)\ \mathrm{d}x.
    \end{equation}
\end{proposition}
To convince the reader that this linear homeomorphism is the right one, we can check that \(x\) and \(p\) are mapped to the correct operators.
\begin{corollary}
    Under the linear homeomorphism from the Schwartz kernel theorem, \(x\) is mapped to \(\phi(x)\mapsto x\phi(x)\) and \(p\) is mapped to \(\phi(x)\mapsto -i\hbar\nabla\phi(x)\).
\end{corollary}
Since states (i.e., elements of \(L^2\) with norm \(1\)) \(\phi\) induce projection operators \(\phi\langle -,\phi\rangle = \lvert\phi\rangle\langle\phi\rvert\), we can also induce its corresponding distribution.
\begin{definition}
    Given a state \(\phi\), its corresponding distribution is called the \emph{Wigner function}, and is given by
    \begin{equation}
        W[\phi](x,p) = \int_{\mathbb{R}^n}\phi\left(x+\frac{y}{2}\right)\overline{\phi\left(x-\frac{y}{2}\right)}e^{-\frac{i}{\hbar}p\cdot y}\ \mathrm{d}y.
    \end{equation}
\end{definition}
One then derives the following properties of the Wigner function.
\begin{proposition}
    For \(\phi(x)\) a state and \(\tilde{\phi}(p) = \frac{1}{(2\pi\hbar)^{n/2}}\int_{\mathbb{R}^n}\phi(x)e^{-\frac{i}{\hbar}p\cdot x}\ \mathrm{d}x\) its normalized Fourier transform, the Wigner function satisfies the following properties:
    \begin{enumerate}
        \item \(W[\phi](x,p)\) is real.
        \item The Wigner function can equivalently be defined on momentum space
        \begin{equation}
            W[\phi](x,p) = \int_{\mathbb{R}^n}\tilde{\phi}\left(p+\frac{q}{2}\right)\overline{\tilde{\phi}\left(p-\frac{q}{2}\right)}e^{\frac{i}{\hbar}q\cdot x}\ \mathrm{d}q.
        \end{equation}
        \item Integrating the Wigner function with respect to momentum or position gives the position or momentum space probability distribution, respectively:
        \begin{equation}
            \lvert\phi(x)\rvert^2 = \frac{1}{(2\pi\hbar)^n}\int_{\mathbb{R}^n}W[\phi](x,p)\ \mathrm{d}p, \qquad \lvert\tilde{\phi}(p)\rvert^2 = \frac{1}{(2\pi\hbar)^n}\int_{\mathbb{R}^n}W[\phi](x,p)\ \mathrm{d}x.
        \end{equation}
        \item For a distribution \(f\) and its corresponding operator \(\hat{f}\),
        \begin{equation}
            \langle\phi\rvert\hat{f}\lvert\phi\rangle = \frac{1}{(2\pi\hbar)^n}\langle W[\phi],f\rangle_{L^2}.
        \end{equation}
    \end{enumerate}
\end{proposition}
\begin{remark}
    While it may be tempting to treat the Wigner function as a probability distribution over phase space, it is actually a \emph{quasi}-probability distribution. This means that it is allowed to be negative in places. For example, the first excitation of a one-dimensional harmonic oscillator,
    \begin{equation*}
        \psi_1(x) = \sqrt{2}\left(\frac{m^3\omega^3}{\pi\hbar^3}\right)^{1/4}xe^{-\frac{m\omega x^2}{2\hbar}}
    \end{equation*}
    has Wigner function
    \begin{equation*}
        W[\psi_1](x,p) = 4\left(\frac{m\omega x^2}{\hbar} + \frac{p^2}{m\omega\hbar} - \frac{1}{2}\right)e^{-\frac{m\omega x^2}{\hbar} - \frac{p^2}{m\omega\hbar}},
    \end{equation*}
    which is negative at \((0,0)\).
\end{remark}
The time evolution of a function is given by
\begin{equation}
    \frac{\mathrm{d}f}{\mathrm{d}t} = \{H,f\}_M = \{H,f\}_P + \mathcal{O}(\hbar^2).
\end{equation}
In particular, we notice that the classical Poisson bracket agrees up to first order in \(\hbar\), rather than to zeroth order, as expected in the classical case. Thus, we can use classical methods to approximate quantum mechanics up to first order \emph{if} we can choose our initial values to also agree up to first order (in \autoref{sec:SED}, we will present an argument that this is no problem in stochastic electrodynamics). There is, however, one essential property of quantum theory which is not captured by ordinary classical mechanics: quantum fluctuations. This means that, even in a vacuum, there are random excitations in the field, with expectation value zero, but nonzero variance. We have experimentally verified the existence of such fluctuations through the Casimir effect. Therefore, a theory that aims to describe quantum theory up to first order should somehow model these quantum fluctuations.\par
It is rather straightforward to generalize Weyl--Wigner quantization to an arbitrary Hilbert space \(X\) instead of \(\mathbb{R}^{2n}\). In this case, we consider a Moyal bracket over \(\mathcal{D}'(X\oplus X)\), which is defined in the same way as above. This is how we quantize fields: if \(X\) is, for example, \(L^2(M,\mathbb{R})\), then this description tells us how to quantize scalar fields. An explicit description of the Wigner function of fields, as induced from their Heisenberg formulation, can be found in \cite[sec. 2.2]{santosAnalogyStochasticElectrodynamics2022}.\par
Due to the setup of Weyl--Wigner quantization, it is also fairly straightforward to generalize the concept to arbitrary (Poisson) manifolds \((M,\{-,-\}_P)\). We simply require there to be some suitable operator \(\{-,-\}_M\) on \(C^\infty(M)[[i\hbar]]\) which agrees with the Poisson bracket in zeroth order. One calls this structure a \emph{deformation quantization}, and it turns out that one exists for all Poisson manifolds, see \cite{kontsevichDeformationQuantizationPoisson2003}. In this paper, we will only ever work with the Poisson bracket, so discussing deformation quantization more is not within scope.\par
In modern quantum field theory, one commonly considers quantum mechanics perturbatively.  That is to say, we calculate quantities as power series in \(\hbar\) and coupling constants. However, we also know that there are some nonperturbative effects, evidenced by the fact that the power series diverges for any nonzero \(\hbar\). Up to these effects, however, perturbative quantum theories, such as quantum electrodynamics (QED), are ``precise''. Since the Poisson bracket agrees with the Moyal bracket up to first order, we can conclude that the Liouville equation must agree with perturbative quantum theory up to first order.
\subsection{Hamiltonian electrodynamics}\label{sec:hamiltonianED}
In this section, we fix our Hamiltonian models of electrodynamics (in Coulomb and Feynman--'t Hooft gauge) for the rest of this work. Electromagnetism fundamentally takes place in three dimensions, so we will only consider a three-torus, and \(n\) particles that live on it, position denoted by \(q_i \in \mathbb{T}^3\) or \(q = (q_i)_{i=1}^n \in \mathbb{T}^{3n}\), with corresponding momenta \(p = (p_i)_{i=1}^n\), and masses \((m_i)_{i=1}^n\). We assume the particles have an electrical charge distribution \(\rho_i(x-q_i)\). Note that particles in this model cannot change shape or rotate.\par
First, as standard in stochastic electrodynamics, we will work in Coulomb gauge, i.e., we have an electromagnetic vector potential field \(A\colon\mathbb{T}^3\to\mathbb{R}^3\) with \(\nabla\cdot A=0\) and its momentum conjugate field \(\Pi\), and the electric scalar potential is fully described by the Coulomb potential. For analytical reasons (see \autoref{rmk:reasonforH12}), we take \(A\in \dot{H}^{1/2}_\perp(\mathbb{T}^3,\mathbb{R}^3) = \{A\in \dot{H}^{1/2}_\perp(\mathbb{T}^3,\mathbb{R}^3)\mid \nabla\cdot A = 0\}\), where \(\dot{H}^k\) denotes the homogeneous Sobolev space
\begin{equation}
    \dot{H}^k(A,B) \coloneqq \{f\in H^{-\infty}(A,B)\mid (-\Delta)^{k/2}f\in L^2(A,B)\}/\ker{(-\Delta)^{k/2}},
\end{equation}
which is a Hilbert space with inner product \(\langle f,g\rangle_{\dot{H}^k} = \langle(-\Delta)^{k/2}f,(-\Delta)^{k/2}g\rangle_{L^2}\). Finally, due to how prevalent integration with respect to \(\rho_i\) will be, we introduce the following shorthand,
\begin{equation}
    \langle f\rangle_{q_i} \coloneqq \langle f(x),\rho_i(x-q_i)\rangle_{L^2}.
\end{equation}
In the point-particle limit \(\rho_i\to \delta\), we get \(\langle f\rangle_{q_i} = f(q_i)\), so one can think of \(\langle f\rangle_{q_i}\) as ``the (weighted average) value of \(f\) at particle \(i\)''.
The Hamiltonian is
\begin{equation}\begin{multlined}
    H(q,p,A,\Pi) = \underbrace{\sum_{i=1}^n\frac{p_i^2}{2m_i} - \sum_{j>i}\frac{1}{\epsilon_0}\langle\Delta^{-1}\rho_j(x-q_j)\rangle_{q_i}}_{H_\text{part}}\\
    + \underbrace{\sum_{i=1}^n\left(-\frac{\sqrt{\mu_0}}{m_i}p_i\cdot\langle A\rangle_{q_i} + \frac{\mu_0}{2m_i}\langle A\rangle_{q_i}^2\right)}_{H_\text{int}} + \underbrace{\frac{1}{2}\lVert (A,c\Pi)\rVert_{\dot{H}^1}^2}_{H_\text{field}}.
\end{multlined}\end{equation}
Note that our vector potential differs by a constant \(\sqrt{\mu_0}\) from the vector potential in physical literature, in order to realize \(\mu_0\) as the coupling between the particles and the field\footnote{Another reason why we choose the units of \(A\) (and later \(\varphi\)) in this way is that we have effectively decoupled the speed of light \(c\) from \(\epsilon_0\) and \(\mu_0\), i.e., we never use (or need) that \(c = \frac{1}{\sqrt{\mu_0\epsilon_0}}\). This means that we may vary the value of \(\mu_0\), as we will do in \autoref{sec:nonmagneticlimit}, without affecting \(c\).}. The kinetic part of the particle Hamiltonian together with the interaction Hamiltonian is \(\sum_{i=1}^n \frac{(p_i - \sqrt{\mu_0}\langle A\rangle_{q_i})^2}{2m_i}\), the electromagnetic kinetic Hamiltonian, and the potential sector of the particle Hamiltonian is the Coulomb potential, since
\begin{equation*}
    -\frac{1}{\epsilon_0}\langle\Delta^{-1}\rho_j(x-q_j)\rangle_{q_i} = \int_{\mathbb{T}^3\times\mathbb{T}^3} \frac{\rho_i(x-q_i)\rho_j(y-q_j)}{4\pi\epsilon_0\lvert x-y\rvert}\ \mathrm{d}x\mathrm{d}y.
\end{equation*}
The second-order equations of motion (in position space) are as expected for electromagnetism in Coulomb gauge,
\begin{gather}
    \square A = -\sqrt{\mu_0}J_\perp,\\
    \ddot{q}_i = \frac{1}{m_i}\left\langle \sqrt{\mu_0}\dot{q}_i\times(\nabla\times A) + \sum_{j\neq i}\frac{1}{\epsilon_0}\nabla\Delta^{-1}\rho_j(x-q_j) - \sqrt{\mu_0}\dot{A} \right\rangle_{q_i},
\end{gather}
where \(\square\coloneqq -\frac{1}{c^2}\frac{\partial^2}{\partial t^2} + \Delta\) is the d'Alembertian or wave operator, and \(J_\perp\) is the projection of the current density to \(\dot{H}^\bullet_\perp\):
\begin{equation}
    J_\perp(x) =  (\Delta - \nabla\nabla\cdot)\Delta^{-1}\sum_{i=1}^n\dot{q}_i\rho_i(x-q_i) = -\frac{1}{4\pi}(\Delta - \nabla\nabla\cdot)\sum_{i=1}^n\dot{q}_i\int\frac{\rho_i(y-q_i)}{|x-y|}\ \mathrm{d}y.
\end{equation}\par
Let us move on to electrodynamics in Feynman--'t Hooft gauge. The Hamiltonian is given by
\begin{equation}\begin{multlined}
    H(q,p,\varphi,\pi,A,\Pi) = \sum_{i=1}^n\left(\frac{(p_i - \sqrt{\mu_0}\langle A\rangle_{q_i})^2}{2m_i} + \frac{1}{\sqrt{\epsilon_0}}\langle\varphi\rangle_{q_i}\right)\\
    + \frac{1}{2}\lVert(A,c\Pi)\rVert_{\dot{H}^1}^2 - \frac{1}{2}\lVert (\varphi,c\pi)\rVert_{\dot{H}^1}^2,
\end{multlined}\end{equation}
where \(\varphi\in\dot{H}^{1/2}(\mathbb{T}^3,\mathbb{R})\) is a scalar field (differing by a constant of \(\frac{1}{\sqrt{\epsilon_0}}\) from physical literature) with momentum conjugate \(\pi\), and \(A,\Pi\in\dot{H}^{1/2}(\mathbb{T}^3,\mathbb{R}^3)\). The second-order equations of motion are
\begin{gather}
    \square{A} = -\sqrt{\mu_0}J,\quad \square{\varphi} = -\frac{1}{\sqrt{\epsilon_0}}\rho,\\
    \ddot{q}_i = \frac{1}{m_i}\left\langle \sqrt{\mu_0}\dot{q}_i\times(\nabla\times A) - \frac{1}{\sqrt{\epsilon_0}}\nabla\varphi - \sqrt{\mu_0}\dot{A}\right\rangle_{q_i}.
\end{gather}
Here, \(J = \sum_{i=1}^N\dot{q}_i\rho_i(x-q_i)\) is the current density and \(\rho(x) = \sum_{i=1}^N\rho_i(x-q_i)\) is the charge density\footnote{Usually, in Feynman--'t Hooft gauge, one would expect the Nakanishi--Lautrup field \(B \coloneqq -\nabla\cdot A - \frac{1}{c}\dot{\varphi}\). It is still there, but hidden; we can derive from Hamilton's equations that \(B = \sqrt{-\Delta}\pi - \nabla\cdot A\).}. 
\subsection{Stochastic electrodynamics}\label{sec:SED}
In this section, we introduce how stochastic electrodynamics (SED) makes classical fields stochastic. SED is a classical approximation to quantum electrodynamics which models vacuum fluctuations of the electromagnetic field as a stochastically determined initial condition known as the zero-point field. In particular, the vacuum fields are real --- not virtual, as is generally assumed. A similar approach was already employed by Einstein and Hopf in \cite{einsteinStatistischeUntersuchungBewegung1910}, but study in the field really took off with Marshall's work in \cite{marshallRandomElectrodynamics1963}.\par
A physical source discussing the construction of SED is \cite{goedeckeStochasticElectrodynamicsStochastic1983}, and an accessible introduction is given in \cite[ch. 10]{cettoQuantumMechanicsPhysical2025}. We will discuss the construction in mathematical terms and high generality, to make the strategy applicable to other field theories.\par
In SED, particles are described by a probability distribution in phase space, and behave more or less classically. However, we do have a stochastic zero-point field. To see where the stochastic part comes in, consider a field \(A\) whose equation of motion is given by
\begin{equation}
    B A(q,t) = \text{(source terms)},
\end{equation}
where \(B\) is some operator with nontrivial kernel (in electrodynamics, \(B\) would be \(\square\)). If we fix some off-shell sector \(\ker{B}^\perp\), its dynamics is determined uniquely by the behavior of the (classical) particles. However, the free part, \(\ker{B}\), is not fixed by the equation of motion, and so may be chosen freely (in electrodynamics, this sector corresponds to radiation). This choice corresponds to a boundary condition: given some initial values of the particle, the initial choice of free field determines the entire dynamics. Classically, the boundary value is often set to zero, arguing that fields without sources --- in other words, in a vacuum --- should vanish. In SED, we employ a random distribution (a Gaussian one, to be precise) in the free sector to account for the zero-point field.\par
In \cite{santosAnalogyStochasticElectrodynamics2022}, it is shown that SED is equivalent to nonrelativistic QED up to first order in \(\hbar\), \emph{given that the initial conditions for the particles agree}. This is in reference to the fact that in SED, we restrict ourselves to probability distributions instead of quasi-probability distributions. However, in \cite{marshallRandomElectrodynamics1963}, it is shown for the harmonic oscillator that the thermal state, a thermodynamic mixture of quantum states, will be a probability distribution, and hence lies in the realm of SED. Marshall furthermore states that Wigner functions which are not probability distributions are to be considered unphysical. This discussion is however not directly relevant to this paper.
\begin{remark}
    It is important to note that the possible initial conditions of QED (up to first order), i.e., quasi-probability distributions, \emph{include} all possible initial conditions of SED. Therefore, by the above results, any (periodic) solution of SED is automatically a periodic solution of QED (up to first order). In this paper, we prove the existence of the prior, which implies the existence of the latter.
\end{remark}
Beyond just an approximation of QED, SED is sometimes posed as a realistic interpretation of quantum mechanics, as discussed in detail in \cite{santosStochasticElectrodynamicsInterpretation2020,santosStochasticInterpretationQuantum2022,santosAnalogyStochasticElectrodynamics2022}. Furthermore, there is also work on SED as an underlying theory behind QED, laid out in \cite{delapenaEmergingQuantumPhysics2015,cettoQuantumMechanicsPhysical2025}. Both lie outside the scope of this paper, and we will simply focus on SED as the closest classical approximation to QED, as per \cite{boyerStochasticElectrodynamicsClosest2019}.\par
In general stochastic particle-field theories, we will take the Hamiltonian to be
\begin{equation}
    H = H_\text{part} + H_\text{int} + H_\text{field},
\end{equation}
where \(H_\text{part},H_\text{field}\) describe (the propagator of) the particles and field, respectively, and \(H_\text{int}\) describes the interaction between the particles and field (particle interactions should be put in \(H_\text{part}\)). The exact form of each of these depends on the precise model we are working with (how many particles, their shape, etc.), but the field Hamiltonian is always assumed to be quadratic. Therefore, as noted in \cite{santosAnalogyStochasticElectrodynamics2022}, since the Moyal bracket depends only on odd derivatives, the quantum equation of motion simplifies to
\begin{equation}
    \{H,f\}_M = \{H_\text{part}+H_\text{int},f\}_M + \{H_\text{field},f\}_P.\label{eq:moyalSimplified}
\end{equation}
This leads us to the following result.
\begin{proposition}\label{prop:sedisexact}
    If both the particle and interaction Hamiltonian are at most quadratic, then the Poisson bracket is equal to the Moyal bracket in all orders.\par
    Furthermore, given an interaction Hamiltonian which is at most linear in the field, then the Poisson bracket is equal to the Moyal bracket in the model with a classical particle and a quantum field (such as the Nelson model).\par
    In particular, in each of the above cases, if the field is stochastically distributed to account for vacuum fluctuations, the stochastic model is equivalent to the corresponding (nonrelativistic) quantum theory, up to initial conditions for the particles.
\end{proposition}
\begin{proof}
    The first part of the proposition directly follows from equation \eqref{eq:moyalSimplified}. As for the second part, consider that terms in the interaction Hamiltonian must depend in some way on the particle. Therefore, the third order derivatives which must vanish will look like either three field-derivatives or one particle-derivative together with two field-derivatives. To make sure they vanish, the field-dependence must be at most linear.
\end{proof}
Let us return to the specific case of SED. The free sector of the field admits the following Fourier expansion,
\begin{equation}
    A_0(x,t) = \sum_{k\in\mathbb{Z}^3\setminus\{0\}}\frac{\pi}{L} \sqrt{\frac{c}{(2\pi)^3\lvert k\rvert }}\left(\alpha_ke^{\frac{2\pi i}{L}(-\lvert k\rvert ct + k\cdot q)} + \alpha_k^*e^{\frac{2\pi i}{L}(\lvert k\rvert ct-k\cdot q)}\right),
\end{equation}
where \(\alpha_k\) is perpendicular to \(k\), denoted \(\alpha_k\in k^\perp\subseteq\mathbb{R}^3\), and has the same units as \(\sqrt{\hbar}\). Each of these is distributed (complex) normally,
\begin{equation}
    \alpha_k \sim \mathcal{CN}_{k^\perp}(\mu(k), 2\sigma(\lvert k\rvert)^2),
\end{equation}
where \(\mu\colon\mathbb{Z}^3\setminus\{0\}\to\mathbb{C}^3\) such that \(\mu(k)\cdot k=0\) and \(\sigma\colon\mathbb{R}_{>0}\to \mathbb{R}_{\geq 0}\). In stochastic electrodynamics, as an approximation to quantum mechanics, \(\mu(k)=0\) and \(\sigma(\lvert k\rvert) = \sqrt{\hbar}\), in order to accurately model the Casimir effect. Note that we have excluded the \(k=0\) portion to avoid division by zero, and because it has no impact on the dynamics.\par
Our model is applicable in more generality than just stochastic electrodynamics. It may be used for any model where the amplitudes of the free wave modes are distributed normally. Such models are known as Gaussian random fields (c.f. \cite{gudderGaussianRandomFields1978}), which have many important applications in, for example, cosmology (\cite{bardeenStatisticsPeaksGaussian1986}), geostatistics (\cite{diggleModelbasedGeostatistics1998}), meteorology (\cite{minakovAcousticWaveformInversion2017}) and radiology (\cite{telschowPreciseFWERControl2024}). They are also central to quantum optics, see \cite{funaiGaussianStatesQuantum2025}, due to the prevalence of harmonic oscillators.\par
If the particles follow \(T\)-periodic trajectories, for \(\frac{c^2T^2}{L^2}\) Diophantine (to be defined below), then the free sector of \(J_\perp\) is zero. In other words, \(\square A_0 = 0\), and no free radiation is produced by the particles. Furthermore,
\begin{equation}
    \Pi_0(x,t) = \sum_{k\in\mathbb{Z}^3\setminus\{0\}}\frac{\pi}{L} \sqrt{\frac{1}{(2\pi)^3\lvert k\rvert c}}\left(-i\alpha_ke^{\frac{2\pi i}{L}(-\lvert k\rvert ct + k\cdot q)} + i\alpha_k^*e^{\frac{2\pi i}{L}(\lvert k\rvert ct-k\cdot q)}\right).
\end{equation}
Therefore, \(A_0\) and \(\Pi_0\) represent the real, respectively imaginary, part of one complex vector field,
\begin{equation}
    A_0 + ic\Pi_0 = \sum_{k\in\mathbb{Z}^3\setminus\{0\}}\frac{2\pi}{L} \sqrt{\frac{c}{(2\pi)^3\lvert k\rvert }}\alpha_ke^{\frac{2\pi i}{L}(-\lvert k\rvert ct + k\cdot q)}\in\dot{H}^{1/2}(\mathbb{T}^3,\mathbb{C}^3).
\end{equation}
From this point on, we will treat the homogeneous field as this complex field and denote it \((A_0,c\Pi_0)\).\par
One of the major (possible) properties of stochastic electrodynamics is that of a stable classical atom. Famously, the classical Rutherford atom model is unstable; as the electron orbits around the nucleus, it is accelerating and so, by the laws of electrodynamics, produces radiation. However, since radiation has positive energy, the electron must in turn lose some kinetic energy. Eventually, the electron must collide with the nucleus. Of course, in reality, this is unphysical: atoms are stable --- otherwise, all matter would be in a constant state of nuclear decay. We needed quantum theory to solve this problem, first through the Bohr model, and later with the quantum mechanical atom. In stochastic electrodynamics, we take the same Rutherford model, but the addition of a stochastic vacuum field adds energy back into the electron. As in \cite{boyerRandomElectrodynamicsTheory1975}, one can then heuristically argue that a stable solution must exist, but no rigorous proof exists.\par
In mathematics, we can prove the existence of periodic solutions in some classical theories using Floer theory. Floer theory has already been used in \cite{fabertCuplengthEstimatesPeriodic2023} to prove the existence of periodic orbits in particle-field systems with fixed zero-point field, and in \cite{fabertCuplengthEstimatesTimeperiodic2024} for periodic orbits of stochastic processes. A natural question is then whether it is possible to extend Floer theory to stochastic particle-field theories, which is what the rest of this paper is dedicated to.
\section{Periodic Stochastic Orbits in the Nonmagnetic Limit}\label{sec:nonmagneticlimit}
To describe how Floer methods may be applied to stochastic field theories, we first consider a simpler particle-field model. The model consists of \(n\) particles that live on the torus \(\mathbb{T}^d = \mathbb{R}^d/(L\mathbb{Z})^d\) with positions \(q = (q_i)_{i=1}^n\) of masses \((m_i)_{i=1}^n\) and charge distribution \(\rho_i(x-q_i)\), and a scalar field \(\varphi(x)\in \dot{H}^{1/2}(\mathbb{T}^d,\mathbb{R})\). Their momentum conjugates are \(p = (p_i)_{i=1}^n\) and \(\pi(x)\), respectively. The Hamiltonian is given by
\begin{equation}
    H(q,p,\varphi,\pi) = \underbrace{\sum_{i=1}^n\frac{\lvert p_i\rvert^2}{2m_i} + F_t(q,p)}_{H_\text{part}} + \underbrace{\vphantom{\frac{1}{2}}\langle\varphi,\rho\rangle_{L^2}}_{H_\text{int}} \underbrace{- \frac{1}{2}\lVert(\varphi,c\pi)\rVert_{\dot{H}^1}^2}_{H_\text{field}},
\end{equation}
where \(F_t\) is \(C^1\)-bounded and \(T\)-periodic for some \(T\), and \(\rho(x) = \sum_{i=1}^n\rho_i(x-q_i)\) describes the total charge distribution.
\begin{remark}\label{rmk:particlefieldfacts}
    We note the following facts:
    \begin{enumerate}
        \item The interaction Hamiltonian is linear in the field, and thus satisfies the latter half of \autoref{prop:sedisexact}. Hence, if the particles are classical and the field quantum, the Poisson bracket is exact up to all orders.
        \item One can make the scalar field massive (of mass \(m\)) by replacing the operator \(-\Delta\) with \(\left(\left(\frac{cm}{\hbar}\right)^2 - \Delta\right)\); in particular, this turns the homogeneous Sobolev space \(\dot{H}^{1/2}\) into the Sobolev space \(H^{1/2}\). Thanks to the closedness of \(\mathbb{T}^d\), these spaces are equivalent (up to constant functions), and hence give the same results.
        \item If we swap the sign of the field Hamiltonian, and give the field a mass of \(\frac{\hbar}{cL}\), we recover the particle-field model in \cite{fabertCuplengthEstimatesPeriodic2023}. The results in this paper hold equally for this model.
    \end{enumerate}
\end{remark}
The corresponding equations of motion are
\begin{align*}
    \dot{q}_i &= \frac{p_i}{m_i} + \nabla_{p_i} F_t(q,p),& \dot{p}_i &=  -\nabla_{q_i} F_t(q,p) - \langle\nabla\varphi,\rho_i(x-q_i)\rangle_{L^2},\\
    \dot{\varphi}(x) &= -(-\Delta)^{1/2}c^2\pi(x),& \dot{\pi}(x) &= (-\Delta)^{1/2}\varphi(x) - (-\Delta)^{-1/2}\rho(x).
\end{align*}
The second-order equation of motion of \(\varphi\) is given by
\begin{equation}
    \square\varphi = -\rho(x).\label{eq:waveeq}
\end{equation}\par
If \(F_t=0\) and \(n=3\), this model describes electrodynamics in Feynman--'t Hooft gauge in the nonmagnetic limit \(\mu_0\to0\) (after absorbing \(\frac{1}{\sqrt{\epsilon_0}}\) into \(\rho\)). Indeed, if we remove the now-decoupled \((A,c\Pi)\)-sector, then our Hamiltonian description of Feynman--'t Hooft gauge, as in \autoref{sec:hamiltonianED}, reduces to the model in this section.\par
After Fourier transforming, the left-hand side of equation \eqref{eq:waveeq} is of the form \((c^{-2}\omega^2 - k^2)\tilde{\varphi}\), which tells us that \(\square\) has a nontrivial kernel. If we define \(\ker{\square}^\perp\) to be all \(\varphi\) such that \(\tilde{\varphi}(k,\omega)=0\) whenever \(c^{-2}\omega^2 = k^2\), we can split \(\varphi = \varphi_0 + \varphi' \in \ker{\square} \oplus \ker{\square}^\perp\). Since the homogeneous part \((\varphi_0,c\pi_0)\) is no longer dynamical, but rather a parameter of the system, we regroup the Hamiltonian to reflect this fact.
\begin{equation}
    H = \underbrace{H(q,p,\varphi_0,\pi_0)}_{F_\text{part}[\varphi_0,\pi_0]} + \underbrace{H_\text{int}(q,p,\varphi',\pi')}_{F_\text{int}} + \underbrace{H_\text{field}(q,p,\varphi',\pi')}_{F_\text{field}}.
\end{equation}
Note that the \(k=0\) sector of \(\varphi_0\) has no impact on the dynamics, so we may set it to zero.
\begin{definition}
    A \emph{stochastic particle-field model}\footnote{Since the interaction term in our Hamiltonian is known in physics as the \emph{Yukawa interaction}, it could be more accurate to call such a model a \emph{stochastic Yukawa model}. We choose to stick with simpler nomenclature in this paper.} is the model described above, where the \(k\neq 0\) Fourier coefficients of \(\varphi_0\) (the \emph{zero-point field}) are normally distributed, and \(\tilde{\varphi}_0(0,0)=0\).\par
    \emph{\(T\)-periodic solutions} to the stochastic particle-field system are \(T\)-periodic measures \(\chi\) over \(C^1(\underbrace{\mathbb{R}/T\mathbb{Z}}_{S^1},\underbrace{T^*\mathbb{T}^{nd} \times (\dot{H}^{1/2}(\mathbb{T}^d)\oplus \dot{H}^{1/2}(\mathbb{T}^d))}_X)\) which have full measure on the functions which satisfy Hamilton's equation, and the pushforward of \(\chi\) to the homogeneous sector is precisely the distribution of the zero-point field.
\end{definition}
\begin{remark}
    To make \(\chi\) behave more like evolving probability distributions (i.e., a Wigner function), we could have instead defined it as a time-periodic measure over \(X\) which satisfies the Poisson or Moyal equation (as a distribution). This is currently incompatible with Floer theory, since it remains unclear how to define the action functional for this case. Our definition is, however, stronger, inducing such a measure through
    \begin{equation*}
        \overline{\chi}(t)(Y\subseteq X) \coloneqq \chi(\{u\in C^1(S^1,X)\mid u(t)\in Y\}). 
    \end{equation*}
\end{remark}
Proving the existence of a periodic solution is difficult, since the machinery of Floer theory does not yet apply to measures on infinite-dimensional spaces. To solve this, we employ the following strategy, based on \cite{fabertCuplengthEstimatesTimeperiodic2024}.
\begin{enumerate}
    \item We discretize the zero-point field's probability distribution to a binomial distribution, and cut it off at some finite frequency  (such that there are only finitely many zero-point field configurations).
    \item We prove that a \(T\)-periodic orbit exists in each of these cases, using Floer methods.
    \item The solutions are combined into a sum of Dirac measures.
    \item By letting the binomial distribution approach a normal distribution, and letting the cutoff go to infinity, we prove that (a subsequence of) the sequence of measures converges.
\end{enumerate}
In \cite{fabertCuplengthEstimatesPeriodic2023}, we see that \(T\)-periodic orbits of non-stochastic particle-field systems exist, as long as \(\frac{c^2T^2}{L^2}\) (the squared ratio between the time and space period) is Diophantine. We summarize the results here, so we can generalize them to the stochastic case. Let us recall what Diophantineness means, and which of its properties are necessary.
\begin{definition}
    A number \(\sigma\in\mathbb{R}\) is \emph{Diophantine} if there exist \(C>0,r\geq 2\) such that
    \begin{equation}
        \inf_{m\in\mathbb{Z}}\left\lvert \sigma - \frac{m}{n} \right\rvert \geq C\lvert n\rvert^{-r}\quad\text{for all \(n\in\mathbb{N}\)}.
    \end{equation}
    In other words, there is a lower bound to how closely \(\sigma\) can be approximated as a rational number.
\end{definition}
This is hardly an assumption, since the set of Diophantine numbers has full measure: any generic choice of \((T,L)\) is Diophantine. In particular, we can approximate any \((T,L)\) as Diophantine to arbitrary accuracy.\par
If \((T,L)\) is Diophantine, free fields can never be \(T\)-periodic, which reveals the paradoxical nature of the stochastic approach: we are looking for \(T\)-periodic solutions to a problem which itself is not \(T\)-periodic. This means that, mathematically, we should instead be looking for solutions which are \(T\)-periodic on the interval \([0,T]\), i.e., stationary points of the time-\(T\) evolution Hamiltonian diffeomorphism --- which we will simply refer to as \(T\)-periodic orbits. Physically, the paradox cancels out, since the phase of the zero-point field is uniformly distributed; therefore, the zero-point field is, in a physical sense, constant.\par
The Diophantineness condition implies that any contributions to the field produced by the particles do not blow up by constructive interference (the small divisor problem), bounding the field.
\begin{lemma}[\cite{fabertCuplengthEstimatesPeriodic2023}]\label{lem:DiophantineEigvals}
    Let \(\phi_t\) be the time evolution corresponding to \(H_\text{field}\). For \(\frac{c^2T^2}{L^2}\) Diophantine, there exists \(C'>0\) such that the distance between any eigenvalue \(e^{2\pi i\frac{cT\lvert k\rvert}{L}}\) (for \(k\in\mathbb{Z}^d\)) of \(\phi_T\) and \(1\) has lower bound
    \begin{equation}
        \lvert e^{2\pi i\frac{cT\lvert k\rvert}{L}} - 1 \rvert \geq C'(\lvert k\rvert)^{1-r}.
    \end{equation}
\end{lemma}
\begin{corollary}[\cite{fabertCuplengthEstimatesPeriodic2023}]\label{cor:DiophantineBounds}
    Under the above conditions, for any \(T\)-periodic field \((\varphi,c\pi)\) and any \(k\),
    \begin{equation}
        \lVert (\varphi,c\pi)\rVert_{\dot{H}^k} \leq \frac{1}{C'}\left(\frac{L}{2\pi}\right)^{r-1}\lVert \phi_T(\varphi,c\pi) - (\varphi,c\pi)\rVert_{\dot{H}^{k+r-1}}.
    \end{equation}
\end{corollary}
\begin{remark}\label{rmk:reasonforH12}
    The time evolution \(\phi_t\) reveals the reason we originally chose to work in \(\dot{H}^{1/2}\) instead of any other \(\dot{H}^k\). To get the same equations of motion on \(\dot{H}^k\) without changing \(H_{\text{int}}\), we must set \(H_{\text{field}} = -\frac{1}{2}\lVert\varphi\rVert_{\dot{H}^1}^2 - \frac{1}{2}\lVert c\pi\rVert_{\dot{H}^{2k}}\), and we find that
    \begin{equation}
        \phi_t(\varphi,c\pi) = \begin{pmatrix}\cos(ct\sqrt{-\Delta}) & \sin(t\sqrt{-\Delta})(-\Delta)^{k-1/2}\\
        -\sin(ct\sqrt{-\Delta})(-\Delta)^{1/2-k} & \cos(ct\sqrt{-\Delta})\end{pmatrix}\begin{pmatrix}\varphi\\c\pi\end{pmatrix},
    \end{equation}
    so \(\phi_t\) is bounded if and only if \(k=\frac{1}{2}\).
\end{remark}
Now we are in a position to prove the existence of orbits. First, due to the noncompactness of our space, we should construct \(C^0\)-bounds, in the sense of \cite{cieliebakPseudoholomorphicCurvesPeriodic1994}.
\begin{lemma}\label{lem:smallness_particlefield}
    Let all the \(\rho_i\) be \(\dot{H}^{r-1}\), the total charge \(\sum_{i=1}^n\int_{\mathbb{T}^d}\rho_i(x)\ \mathrm{d}x\) zero\footnote{This is a standard assumption in electrodynamics, since if the total charge on a closed manifold is nonzero, the electric field cannot be defined globally. This is known as a type II paradox in \cite{liElectrodynamicsCosmologicalScales2016}.}, \(\mathcal{A}'\in\mathbb{R}\) and \(\Phi\) a set of free fields with bounded \(\dot{H}^1\)-norm. There exist \(R_1,R_2>0\) such that for all \((\varphi_0,c\pi_0)\in\Phi\), every \(T\)-periodic orbit of
    \begin{equation}
        \overline{H} \coloneqq F_\text{part}[\varphi_0,\pi_0] + \underbrace{\chi_{R_1}(\lvert p\rvert)\chi_{R_2}(\lvert\langle\varphi',\rho\rangle_{L^2}\rvert)F_\text{int}}_{\overline{F}_\text{int}} + F_\text{field}
    \end{equation}
    with action at most \(\mathcal{A}'\) is also a \(T\)-periodic orbit of \(H\). Here, \(\chi_R(r) = \begin{cases}1\quad&r\leq R/2,\\0\quad&r\geq R\end{cases}\) is smooth with \(-\frac{4}{R}\leq \chi_R'(r)\leq 0\). Furthermore, the field norm can be bounded by \(\lVert(\varphi',c\pi')\rVert_{\dot{H}^k} \leq \sqrt{\frac{5cT}{2C'}\left(\frac{L}{2\pi}\right)^{1-2k}R_2}\) for \(k\leq 1\).
\end{lemma}
The main difference between this result and the results of \cite{fabertCuplengthEstimatesPeriodic2023} is that we consider a set of bounded free fields instead of a single fixed free field.
\begin{proof}
    Since the total charge is zero, the equation of motion of \(\varphi\) tells us that \(\int_{\mathbb{T}^d}\varphi(x)\ \mathrm{d}x = 0\). Therefore, for all \(k\leq\ell\),
    \begin{equation*}
        \lVert\varphi\rVert_{\dot{H}^k} \leq \left(\frac{L}{2\pi}\right)^{\ell-k}\lVert\varphi\rVert_{\dot{H}^\ell},\quad \lVert\rho\rVert_{\dot{H}^k} \leq \left(\frac{L}{2\pi}\right)^{\ell-k}\lVert\rho\rVert_{\dot{H}^\ell}.
    \end{equation*}\par
    By \autoref{cor:DiophantineBounds}, we have for all \(k\in\mathbb{R}\) and \(T\)-periodic orbits \((\varphi',c\pi')\) of \(\overline{H}\)
    \begin{align*}
        \lVert (\varphi',c\pi')\rVert_{\dot{H}^k} &\leq \frac{1}{C'}\left(\frac{L}{2\pi}\right)^{r-1}\lVert \phi_T(\varphi',c\pi') - (\varphi',c\pi')\rVert_{\dot{H}^{k+r-1}},
        \intertext{using \(T\)-periodocity, and realizing that \(\phi_T\) already takes care of time evolution due to \(F_\text{field}\) (c.f. \cite{fabertCuplengthEstimatesPeriodic2023} for details),}
        &= \frac{c}{C'}\left(\frac{L}{2\pi}\right)^{r-1}\left\lVert\int_0^{-T}\nabla^{\dot{H}^{1/2}}\overline{F}_\text{int}\circ\phi_{-t}\ \mathrm{d}t\right\rVert_{\dot{H}^{k+r-1}}\\
        &\leq \frac{cT}{C'}\left(\frac{L}{2\pi}\right)^{r-1}\sup\lVert\nabla^{\dot{H}^{1/2}}\overline{F}_\text{int}\rVert_{\dot{H}^{k+r-1}}.
    \end{align*}
    If we let \(\epsilon\coloneqq n\sup_{i=1,\dots,n}{\lVert\rho_i\rVert_{\dot{H}^{r-1}}}\), such that \(\lVert\rho\rVert_{\dot{H}^{r-1}}\leq \epsilon\), we have for \(k\leq1\)
    \begin{align*}
        \lVert\nabla^{\dot{H}^{1/2}}\overline{F}_\text{int}\rVert_{\dot{H}^{k+r-1}} < 5\left(\frac{L}{2\pi}\right)^{1-k}\epsilon,
    \end{align*}
    so for the same \(k\),
    \begin{equation*}
        \lVert (\varphi',c\pi')\rVert_{\dot{H}^k} \leq \frac{5cT}{C'}\left(\frac{L}{2\pi}\right)^{r-k}\epsilon.
    \end{equation*}
    In particular,
    \begin{equation*}
        \lvert\langle\varphi',\rho(x-q)\rangle_{L^2}\rvert \leq \left(\frac{L}{2\pi}\right)^{r-1}\lVert \varphi'\rVert_{L^2}\lVert\rho\rVert_{\dot{H}^{r-1}} \leq \frac{5cT}{C'}\left(\frac{L}{2\pi}\right)^{2r-1}\epsilon^2.
    \end{equation*}
    We choose \(R_2\coloneqq \frac{10cT}{C'}\left(\frac{L}{2\pi}\right)^{2r-1}\epsilon^2\).\par
    Recall that the symplectic action functional (on a cotangent bundle) of a \(T\)-periodic curve \(u\) is given by
    \begin{equation}
        \mathcal{A}(u) = \int_0^T p(t)\cdot\dot{q}(t) - H(u(t))\ \mathrm{d}t.\label{eq:actiondefinition}
    \end{equation}
    Thus, due to the action limitation, given that the \(C^0\)-upper bound on \(F_t\) is some \(V\), and the \(\dot{H}^1\)-norm of the zero-point fields is bounded by \(\mathbb{A}\), if \(M = \sup_{i=1,\dots,n}{m_i}\),
    \begin{align*}
        \frac{\lvert p\rvert_{L^2}^2}{2M} &\leq \mathcal{A}' + TV + T\mathbb{A}\left(\frac{L}{2\pi}\right)^r\epsilon + \left(\frac{L}{2\pi} + \frac{5cT}{2C'}\right)\frac{5cT^2}{C'}\left(\frac{L}{2\pi}\right)^{2r-2}\epsilon^2 + \frac{T}{2}\mathbb{A}^2,
        \intertext{and since we may let \(\mathbb{A}\) be arbitrarily large,}
        &\leq T\mathbb{A}^2.
    \end{align*}
    If we let \(V'\) be the \(C^1\)-upper bound of \(F_t\), the equation of motion of \(p_i\) gives
    \begin{align*}
        \lvert \dot{p}\rvert_{L^2} &\leq T^{1/2}V' + \left(\mathbb{A} + \frac{5cT}{C'}\left(\frac{L}{2\pi}\right)^{r-1}\epsilon\right)T^{1/2}\left(\frac{L}{2\pi}\right)^{r-1}\epsilon,
        \intertext{again using that \(\mathbb{A}\) may be arbitrarily large,}
        &\leq 2T^{1/2}\left(\frac{L}{2\pi}\right)^{r-1}\mathbb{A}\epsilon.
    \end{align*}
    Therefore, in particular, for \(\mathbb{A}\) sufficiently large,
    \begin{equation*}
        \lvert p\rvert \leq T^{-1/2}\lvert p\rvert_{L^2} + T^{1/2}\lvert\dot{p}\rvert_{L^2} \leq \left(\sqrt{2M} + 2T\left(\frac{L}{2\pi}\right)^{r-1}\epsilon\right)\mathbb{A}.
    \end{equation*}
    We choose \(R_1 \coloneqq 2\left(\sqrt{2M} + 2T\left(\frac{L}{2\pi}\right)^{r-1}\epsilon\right)\mathbb{A}\). This completes the proof.
\end{proof}
The lemma applies also for the general particle-field system in \cite{fabertCuplengthEstimatesPeriodic2023}. An essential assumption here is that the interaction Hamiltonian has bounded first derivative, which is not the case in electrodynamics outside the nonmagnetic limit.\par
We finally conclude with the existence of stochastic orbits.
\begin{makethm}{theorem}{thm:cuplengthnonmagnetic}
    Consider a stochastic particle-field model on \(\mathbb{T}^d = \mathbb{R}^d/(L\mathbb{Z})^d\) with \(n\) particles with charge distributions \((\rho_i)_{i=1}^n\), and \(T>0\). If \(\frac{c^2T^2}{L^2}\) is Diophantine, the \(\rho_i\) are \(\dot{H}^{r-1}\), and the total charge \(\sum_{i=1}^n\int_{\mathbb{T}^d}\rho_i(x)\ \mathrm{d}x\) is zero, then there exist at least \((nd+1)\) \(T\)-periodic solutions.
\end{makethm}
\begin{proof}
    In this proof, we will set \(n=1\) for clarity, but remark that it is generalized by simply replacing \(\mathbb{T}^d\) with \(\mathbb{T}^{nd}\). \par
    To prove this, we essentially follow the strategy outlined above. We consider any discretization of the zero-point field, and cut it off at a finite frequency, such that the set of zero-point field configurations \(\Phi\) is finite and contained in \(\dot{H}^1\), and is equipped with the probability function \(\nu\colon\Phi\to[0,1]\). By \autoref{lem:smallness_particlefield} and \cite[Theorem 4.1]{fabertCuplengthEstimatesPeriodic2023}, \((d+1)\) orbits exist for each \((\varphi_0,c\pi_0)\in\Phi\) (thanks to the \(C^0\)-bounds of the cut-off Hamiltonian). The construction of these solutions follows standard Floer methods, and as such, we will only sketch the highlights (see \cite{fabertCuplengthEstimatesPeriodic2023,fabertCuplengthEstimatesTimeperiodic2024} for more detailed implementations of the strategy).\par
    Hamiltonian Floer theory concerns a variant of Morse theory on the space of \(T\)-periodic curves on, in this case, \(X\). Periodic solutions to Hamilton's equations are precisely stationary points of the action \(\mathcal{A}(u)\) (see equation \eqref{eq:actiondefinition}). Upwards gradient flows \(\tilde{u}(s,t)\colon\mathbb{R}\times S^1\to X\), known as \emph{Floer curves} or \emph{Floer cylinders}, connect these solutions, and satisfy the Floer equation
    \begin{equation}
        \partial_s\tilde{u} - J(\partial_t\tilde{u} - X_H(\tilde{u}(s,t))) = 0,
    \end{equation}
    for \(X_H\) the Hamiltonian vector field of \(H\). The energy of a Floer curve is given by
    \begin{equation}
        E(\tilde{u}) = \int_{-\infty}^\infty\int_0^T \lvert\partial_s\tilde{u}(s,t)\rvert^2\ \mathrm{d}t\mathrm{d}s.
    \end{equation}
    If a Floer curve connects periodic solutions \(\lim_{s\to\pm\infty}\tilde{u}(s,-) = u^\pm\), the energy is equal to \(\mathcal{A}(u^+) - \mathcal{A}(u^-)\).\par
    We furthermore define a moduli space \(\mathcal{M}^{\tau,\leq\mathcal{A}'}_{(\varphi_0,c\pi_0)}\) which consists of contractible finite-energy Floer cylinders of \(F_{\text{part}}[\varphi_0,\pi_0]+F_{\text{field}}\) of action at most \(\mathcal{A}'\), with \(\overline{F}_{\text{int}}\) turned on in the middle (with \(R_1,R_2\) induced from \autoref{lem:smallness_particlefield}) --- that is, it satisfies the modified Floer equation
    \begin{equation*}
        \partial_s\tilde{u} - J(\partial_t\tilde{u} - X_{F_{\text{part}}[\varphi_0,\pi_0]+\sigma_\tau(s)\overline{F}_{\text{int}}+F_{\text{field}}}) = 0,
    \end{equation*}
    where \(\sigma_\tau\colon\mathbb{R}\to[0,1]\) is smooth with \(\sigma_\tau(s) = \begin{cases}0\quad&s\leq-1, s\geq\tau+1,\\1\quad&0\leq s\leq\tau\end{cases}\).\par
    The (Gromov--Floer) compactification of \(\mathcal{M}^{\tau,\leq\mathcal{A}'}_{(\varphi_0,c\pi_0)}\) is nonempty. Furthermore, for \(\mathcal{A}'\) large enough, the evaluation map pulls back \(\theta_1\smile\dots\smile\theta_d\), where \(\theta_i\) are some generators of the cohomology of \(\mathbb{T}^d\), to a nonzero cohomology class for all \(\tau\geq 0\) (see \cite[Proposition 4.4]{fabertCuplengthEstimatesPeriodic2023} and \cite[Theorem 7.6]{cieliebakPseudoholomorphicCurvesPeriodic1994}). By taking the limit \(\tau\to\infty\), we know that there are Floer curves \((\tilde{u}_j(\varphi_0,c\pi_0))_{j=1}^d\) of \(\overline{H}[\varphi_0,c\pi_0]\) which satisfy the following properties:
    \begin{enumerate}
        \item The curve \(\tilde{u}_j(\varphi_0,c\pi_0)\) connects a periodic solution \(u^-_j(\varphi_0,c\pi_0)\) to another solution \(u^+_j(\varphi_0,c\pi_0)\) for each \(j\) (due to compactness of the moduli space) of action at most \(\mathcal{A}'\). Therefore, by \autoref{lem:smallness_particlefield}, they are also periodic solutions of \(H\).
        \item For some representatives \(C_j\) of the Poincaré dual of \(\theta_i\), \(\pi_{\mathbb{T}^d}(\tilde{u}_j(\varphi_0,c\pi_0)(0,0))\in C_j\).
        \item\label{item:actionineq} The actions of the solution satisfy the inequalities
        \begin{equation*}
            \mathcal{A}(u_1^-(\varphi_0,c\pi_0)) < \mathcal{A}(u_1^+(\varphi_0,c\pi_0)) \leq \mathcal{A}(u_2^-(\varphi_0,c\pi_0)) < \dots < \mathcal{A}(u_d^+(\varphi_0,c\pi_0)).
        \end{equation*}
    \end{enumerate}
    We combine the solutions into measures
    \begin{equation*}
        \chi_j^\pm(\Phi)(u) = \sum_{(\varphi_0,c\pi_0)\in\Phi} \delta(u-u_j^\pm(\varphi_0,c\pi_0))\nu(\varphi_0,c\pi_0).
    \end{equation*}\par
    By choosing a particular sequence \((\Phi_n, \nu_n)_{n\in\mathbb{N}}\) which converges towards the continuous zero-point field distribution, we get a sequence of measures \((\chi_{j,n}^\pm)_{n\in\mathbb{N}}\). By \autoref{lem:convergenceofmeasure} below, we find that (a subsequence of) the measures must converge to a measure \(\chi_j^\pm\). These measures \(\chi_j^\pm\) are precisely solutions to the stochastic particle-field system. All that remains to prove is that at least \((d+1)\) of these are distinct. We will do this by proving that their actions,
    \begin{equation*}
        \mathcal{A}(\chi_j^\pm) \coloneqq \int_{C^1(S^1,X)} \mathcal{A}(u)\ \mathrm{d}\chi_j^\pm(u)
    \end{equation*}
    are distinct.
    \par
    To prove distinctness of the actions, we can similarly combine the Floer curves into measures \(\tilde{\chi}_j(\Phi)\) over \(C^1(\mathbb{R}\times S^1,X)\). By \autoref{lem:convergenceofcylindermeasure} below, some subsequences of these must converge to measures \(\tilde{\chi}_j\) on \(C^1([-S,S]\times S^1,X)\) for arbitrary \(S>0\).\par
    Let us use the Floer cylinder measure to prove that the actions are distinct. We know that the energy of a Floer cylinder is equal to the difference of the actions of the limits, so, since we only consider the \([-S,S]\) portion of the Floer cylinder, we know that
    \begin{equation*}
        \int_{C^1([-S,S]\times\mathbb{R},X)} \int_{-S}^S\int_{S^1} \lvert\partial_s\tilde{u}(s,t)\rvert^2\ \mathrm{d}t\mathrm{d}s\mathrm{d}\tilde{\chi}_j(\tilde{u}) \eqcolon E(\tilde{\chi}_j) \leq \mathcal{A}(\chi_j^+) - \mathcal{A}(\chi_j^-).
    \end{equation*}
    Since the left-hand side of the equation is clearly nonnegative, we find that \(\mathcal{A}(\chi_j^+) \geq \mathcal{A}(\chi_j^-)\). All that remains to do is prove that this inequality is strict. Indeed, if the energy of \(\tilde{\chi}_j\) were zero for some \(j\), we know that it has full measure on the portion of functions which are constant with respect to \(s\) on \([-S,S]\).  However, this would imply that some \(\tilde{\chi}_{j,n}\) have nonzero measure on this sector, and therefore, there would be \((\varphi_0,c\pi_0)\) such that \(\partial_s\tilde{u}_j(\varphi_0,c\pi_0)=0\) on \([-S,S]\). By the unique continuation principle, this must hold also on \(\mathbb{R}\), and therefore, \(E(\tilde{u}_j(\varphi_0,c\pi_0))=0\), which gives a contradiction with property (\ref{item:actionineq}) of our periodic orbits. Therefore, the energy is positive, and we conclude
    \begin{equation*}
        \mathcal{A}(\chi_1^-) < \mathcal{A}(\chi_1^+) \leq \mathcal{A}(\chi_2^-) < \dots < \mathcal{A}(\chi_d^+).
    \end{equation*}
    There are therefore at least \(d+1\) solutions. This concludes the proof.
\end{proof}
Let us prove the lemmata used in the last proof.
\begin{lemma}\label{lem:convergenceofmeasure}
    Let \(\chi_n\coloneqq \chi_{j,n}^\pm\) for some \(j,\pm\). The sequence \((\chi_n)_{n\in\mathbb{N}}\) is asymptotically tight. In particular, it admits a subsequence which converges (weakly) to some Borel measure.
\end{lemma}
\begin{proof}
    This proof is based on \cite[Lemma 3.2]{fabertCuplengthEstimatesTimeperiodic2024}. Since each \(\Phi\) is contained in \(\dot{H}^1\), and the sequence is convergent, there must exist for any \(\epsilon>0\) an \(\mathbb{A}\) such that
    \[\liminf_{n\to\infty} \nu_n(\dot{H}^1_{\lVert\cdot\rVert_{\dot{H}^1}\leq\mathbb{A}}(\mathbb{T}^d,\mathbb{R}))\geq 1-\epsilon.\]\par
    As in \cite[Proposition 4.4]{fabertCuplengthEstimatesPeriodic2023}, and using \cite[Proposition 9.1.4]{mcduffJholomorphicCurvesSymplectic2012}, we know the action of the orbits of \(\chi_n\) whose vacuum fields are in \(\dot{H}^1_{\lVert\cdot\rVert_{\dot{H}^1}\leq\mathbb{A}}(\mathbb{T}^d,\mathbb{R})\) are bounded by some function which is at most quadratic in \(\mathbb{A}\) plus \(4T\lvert \overline{F}_\text{int}\rvert_{C^0}\). By the proof of \autoref{lem:smallness_particlefield}, there exists a global upper bound for all vacuum fields. Thus, by the same lemma, there are corresponding \(C^0\)-bounds \((R_1,R_2)\) of all of the orbits.
    Since the \(\dot{H}^1\)-norm is bounded (as per \autoref{lem:smallness_particlefield}), these induce a subset \(C_{R_1,R_2}\) of \(T^*\mathbb{T}^d\times(\dot{H}^{1/2}(\mathbb{T}^d)\oplus \dot{H}^{1/2}(\mathbb{T}^d))\) with measure \(\liminf_{n\to\infty}\chi_n(C_{R_1,R_2})\geq 1-\epsilon\), which is compact by the Kondrachov embedding theorem. Therefore, \((\chi_n)_{n\in\mathbb{N}}\) is asymptotically tight. By the Prokhorov--Le Cam theorem, the measure sequence is precompact, and thus has a convergent subsequence.
\end{proof}
The same results hold for the Floer curve measures.
\begin{lemma}\label{lem:convergenceofcylindermeasure}
    We restrict ourselves to values of \(s\) in \([-S,S]\), and let \(\tilde{\chi}_n\coloneqq \chi_{j,n}\) for some \(j\). The sequence \((\tilde{\chi}_n)_{n\in\mathbb{N}}\) is asymptotically tight. In particular, it admits a subsequence which (weakly) converges to some Borel measure.
\end{lemma}
\begin{proof}
    To prove this, we should prove the existence of \(C^0\)-bounds on the Floer cylinders. The rest then follows as in \autoref{lem:convergenceofmeasure}.\par
    Let \((\varphi',c\pi')\) be any field solution which makes up the discrete measures. By the Floer equation, we have on the Floer cylinder for \(k\leq 1\),
    \begin{align*}
        \lVert(\partial_s + J\partial_t)(\phi_{-t}\circ(\varphi',c\pi'))\rVert_{L^2([-S,S]\times S^1,\dot{H}^{k+r-1})} &\leq \sqrt{2ST}\sup{\left\lVert\nabla^{\dot{H}^{1/2}}\overline{F}_{\text{int}}\right\rVert_{\dot{H}^{k+r-1}}}\\
        &\leq \sqrt{2ST}5\left(\frac{L}{2\pi}\right)^{1-k}\lVert\rho\rVert_{\dot{H}^{r-1}}.
    \end{align*}
    Therefore, by \cite[Lemma 4.9]{fabertFloerHomologyHamiltonian2021} (after extending the cylinder by letting the field go to zero far away outside \([-S,S]\)), we have that
    \begin{equation*}
        \lVert(\varphi',c\pi')\rVert_{H^1([-S,S]\times S^1,\dot{H}^k)} = \lVert\phi_{-t}\circ(\varphi',c\pi')\rVert_{H^1([-S,S]\times S^1,\dot{H}^k)}
    \end{equation*}
    is bounded for \(k<1\).\par
    Assume that there is no \(C^0([-S,S]\times S^1,\dot{H}^{1/2})\)-bound for \(\partial_s(\varphi',c\pi')\). Then there exists a sequence such that the energy density at some point goes to infinity: the cylinder bubbles off a holomorphic sphere. Due to the exactness of the symplectic form, this is impossible. Hence, a \(C^0\)-bound must exist. By the Floer equation, the same must hold for \(\partial_t(\varphi',c\pi')\).\par
    Combining the bounds, we find \(C^1([-S,S]\times S^1,\dot{H}^{1/2})\)-bounds (and, therefore, \(C^0\)-bounds) on \((\varphi',c\pi')\).
    Furthermore, by \cite[Theorem 5.4]{cieliebakPseudoholomorphicCurvesPeriodic1994}, we also have bounds for the momentum on the Floer cylinder. The proof then follows as in \autoref{lem:convergenceofmeasure}.
\end{proof}
\section{Outlook}\label{sec:C0bounds}
As the first step in constructing stochastic orbits in general SED (outside the nonmagnetic limit), we should prove something akin to \autoref{lem:smallness_particlefield}. However, as remarked there, this will not work, due to some of the quadratic terms in the interaction Hamiltonian. Let us sketch how this relates to the proof method.\par
Let \(K^\bullet\in\{H^\bullet,\dot{H}^\bullet\}\) on a space of any dimension, and consider a polynomial interaction Hamiltonian of the form
\begin{equation}
    H_\text{int} = gp^m\langle\varphi,\rho(x-q)\rangle_{L^2}^n,\label{eq:generalHint}
\end{equation}
with \(\varphi\in K^{1/2}\), \(m\geq 0\), \(n>0\) integers, and \(g\in\mathbb{R}\). We also define \(\epsilon \coloneqq \lVert\rho\rVert_{K^{r-1}}\) (and if \(K^\bullet = \dot{H}^\bullet\), we assume that \(\rho\) integrates to zero). For simplicity, we work with only one charged particle, but our arguments also generalize to several particles. The statement we want to show is the following.
\begin{claim}
    For \(\rho\) satisfying some requirements and the action bounded by some \(\mathcal{A}'\), for all finite sets of zero-point fields, there exist \(R_1,R_2\) such that all \(T\)-periodic orbits of the \((R_1,R_2)\)-cutoff Hamiltonian are also \(T\)-periodic orbits of the original Hamiltonian.
\end{claim}
\autoref{lem:smallness_particlefield} is, of course, an example of such a statement, where the restriction on \(\rho\) is that it is \(\dot{H}^{r-1}\). We will attempt the same proof strategy, and simply consider how each of the relevant bounds scales with \(R_1,R_2,\mathbb{A}\). For \(k\leq 1\),
\begin{equation*}
    \lVert\nabla^{K^{1/2}}\overline{F}_\text{int}\rVert_{K^{k+r-1}} \propto \epsilon R_1^m(R_2+\mathbb{A}\epsilon)^{n-1},
\end{equation*}
such that, when we use that a variant of \autoref{cor:DiophantineBounds} also holds for \(H^\bullet\),
\begin{equation*}
    \lVert (\varphi',c\pi')\rVert_{K^k} \propto \epsilon R_1^m(R_2+\mathbb{A}\epsilon)^{n-1}.
\end{equation*}
Thus, \(R_2 \propto \epsilon^2 R_1^m(R_2+\mathbb{A}\epsilon)^{n-1}\), which, since \(R_1\) must grow with \(\mathbb{A}\), only allows solutions for arbitrarily large \(\mathbb{A}\) if \(n\in \{1,2\}\) (for \(\epsilon\) small enough).\par
If \(n=1\), we find \(R_2\propto \epsilon^2R_1^m\). Therefore, \(F_\text{int}\propto R_1^m(\epsilon^2R_1^m + \mathbb{A}\epsilon)\), and \(F_\text{field}\propto (\epsilon R_1^m + \mathbb{A})^2\). It follows that
\begin{equation*}
    \lvert p\rvert \propto \epsilon^2R_1^{2m} + \mathbb{A},
\end{equation*}
so \(R_1\propto \epsilon^2R_1^{2m} + \mathbb{A}\). Since \(\mathbb{A}\) may grow arbitrarily large, there comes a point where, for \(m\geq 1\), there are no more solutions. Therefore, we must impose \(m=0\) for \(C^0\)-bounds to hold.\par
If \(n=2\), the growing of \(R_1\) once again introduces a restriction to defining \(R_2\) in arbitrary values: \(m=0\), such that \(R_2\propto\mathbb{A}\epsilon^2\). It then follows that \(\lvert p\rvert \propto \mathbb{A}^2\), and thus, \(R_1\) can be defined for any value of \(\mathbb{A}\).\par
We conclude: the proof strategy only applies when \(m=0\) and \(n\in\{1,2\}\). This proves problematic for the existence of orbits in SED: the \(p_i\cdot\langle A\rangle_{q_i}\) term makes it so that for large enough \(\mathbb{A}\), there are no solutions for \(R_1,R_2\), unless \(\epsilon\) is shrunk arbitrarily small. However, there is a relation between the proof strategy and regularization that may provide hope still. \par 
Both classical electrodynamics and QED have divergences caused by the same origin: electron self-interaction through the electromagnetic field. To get rid of this divergence, one must impose a regularization, i.e., some parametrization of the divergence. In (relativistic) QFTs, one uses dimensional regularization, but for our purpose, we use a (nonrelativistic) momentum cutoff. This means that we introduce a parameter \(\Lambda\) which cuts off the momentum domain in any integrals. This regularized electrodynamics admits finite results, which diverge only for \(\Lambda\to\infty\). By getting rid of the divergent terms, one gets a renormalized theory, which is what we consider to be physical. A famous argument for renormalization is that we do not know what happens at very low scales (equivalently, high momentum scales), so we should only consider theories at measurable scales: those regularized by a value \(\Lambda\).\par
In our proof of \(C^0\)-bounds, we fix the distance scale (through \(\rho\)), and the maximal action, and prove that there exist momentum cutoffs \((R_1,R_2)\) such that orbits agree. In other words, the theory is regularized in such a way that preserves orbits! Thus, after regularization, the theory admits periodic orbits, since the \(C^0\)-bounds hold trivially, though these orbits may not be orbits of the unregularized Hamiltonian. This gives us the hunch that the renormalized theory also admits (stochastic) periodic orbits. To prove this, we would require a broader framework to define the convergence of regularized orbits to a renormalized one.
\printbibliography[heading=bibintoc]
\end{document}